\documentclass[12pt, fceqn]{article}
\usepackage{amsfonts, amsthm, latexsym, amssymb}

\usepackage[fleqn]{amsmath}
\usepackage{graphicx}
\usepackage{graphics}
\usepackage[comma, numbers]{natbib} 
\usepackage[comma]{natbib} 
\usepackage{color}
\usepackage{comment}
\usepackage{algorithm}
\usepackage{algorithmic}
\usepackage{lscape}
\usepackage{footnote}
\usepackage{algorithm,caption}

\usepackage{epsfig}
\usepackage{fancybox}
\usepackage{psfrag}
\usepackage{tabularx}
\usepackage{float}
\usepackage{multirow}
\usepackage{rotating}
\usepackage{cancel}
\usepackage{epstopdf}
\usepackage{booktabs}

\date{}

\DeclareMathOperator*{\Min}{\mbox{Min }}

\newcommand{\om}{\omega}
\renewcommand{\O}{\Omega}

\newtheorem{thm}{THEOREM}[section]

\newtheorem{prop}[thm]{PROPOSITION}
\theoremstyle{remark}

\theoremstyle{definition}
\newtheorem{defn}[thm]{DEFINITION}

\renewcommand{\O}{\Omega}

\newenvironment{hanglist} 
  {\begin{list}{}{\setlength{\itemsep}{4pt}
  \parindent 0pt  \setlength{\parsep}{0pt}\setlength{\leftmargin}{+25pt}
  \setlength{\itemindent}{-\parindent}}}{\end{list}}

\begin{document}
	
\begin{center}
	{\Large Absolute semi-deviation risk measure for ordering problem with transportation cost in Supply Chain}\\[12pt]

	\footnotesize
	
	\mbox{\large Saravanan Venkatachalam$^\text{a}$$^\dagger$, Lewis Ntaimo$^\text{b}$}\\
	$^\text{a}$Department of Industrial and Systems Engineering, Wayne State University, \\ 4815 Fourth St., Detroit, MI 48202, USA. \\	
	$^\text{b}$Department of Industrial and Systems Engineering, Texas A\&M
	University, \\ 3131 TAMU, College Station, TX 77843, USA. \\
	$^\dagger$Corresponding author: \mbox{saravanan.v@wayne.edu}\\
	\normalsize
	
\end{center}

\begin{abstract}
	We present a decomposition method for stochastic programs with 0-1 variables in the second-stage with absolute semi-deviation (ASD) risk measure. Traditional stochastic programming models are risk-neutral where expected costs are considered for the second-stage. A common approach to address risk is to include a dispersion statistic in addition with expected costs and weighted appropriately. Due to the lack of block angular structure, stochastic programs with ASD risk-measure possess computational challenges. The proposed decomposition algorithm uses another risk-measure `expected excess', and provides tighter bounds for ASD stochastic models. We perform computational study on a supply chain replenishment problem and standard knapsack instances. The computational results using supply chain instances demonstrate the usefulness of ASD risk-measure in decision making under uncertainty, and knapsack instances indicate that the proposed methodology outperforms a direct solver.\\
	
\end{abstract}
	
\section{Introduction}\label{sec-intro}
A two-stage stochastic programming (SP2) approach consists of `here-and now' first-stage decisions and recourse decisions in the second-stage. Choosing the first-stage decision is considered as selecting a random variable based on the recourse costs. In  traditional two-stage stochastic programs, risk neutral expected costs are considered in the second-stage which implies that first-stage random variable is selected only based on the mean. There are wide range of applications for SP2 in the literature with the expected costs however when the random data has a significant variance then a solution based on expected costs alone may not provide robustness. To alleviate this, a weighted risk term is added along with the mean, and this is known as two-stage stochastic programs with risk-measures. When risk measures are coherent, then they are known to possess suitable mathematical properties conducive to construct to efficient algorithms. Two broad categories of coherent risk terms are introduced; \textit{quantile} and \textit{deviation} based risk measures. \textit{Quantile} risk-measures are based on the quantiles of the probability distributions of the random data. Types of quantile based risk-measures include: \textit{excess probability} - measures the probability of exceeding a prescribed target level, \textit{conditional-value-at-risk} (CVaR) - measures the expectation of worst outcomes for a given probability.  Deviation risk measures are based on the deviation of first-stage random variable from a prescribed target, and deviation based risk measures include: \textit{expected excess} (EE) - measures the expected value of the excess over a given target, \textit{absolute semi-deviation} (ASD) - measures the expected value of the excess over the mean.  In this work, we study and present a methodology for solving large scale SP2 with absolute semi-deviation risk measure. The computational efficacy of the methodology is evaluated using multi-items knapsack instances from the literature, and as a supply chain application, multi-item single source ordering problem (MSSOP) with transportation cost is presented with ASD risk-measure, and solutions are compared with risk-neutral expected costs.

Incorporating risk measures in SP2 brings up computational challenges and is a fairly a new research area. Except ASD, all other risk measures described above have block angular structure so they are conducive for standard Benders' decomposition based L-shaped method which is primarily used for solving SP2. However, other risk measures depend upon a user defined target apart from weighted risk term. The optimal \textit{user-defined} targets are difficult to determine in practice unless the user has an excellent knowledge about the application and uncertainty in the data. In contrast, ASD does not require a user defined target for the model however it lacks block angular structure necessary for decomposition of scenarios in L-shaped method. In this study, we propose a methodology to use a modified representation of EE to solve large scale SP2 with ASD risk measures. 

More than one item is purchased from a vendor on a periodic basis in many supply chains (SCs). The existence of multiple items in a procurement process poses a challenge as these items may face distinct demand streams and have different shipment and weight constraints, and variable holding rates. Typically, a vendor or a transportation provider may take advantage by purchasing a group of items through transportation discounts. However, this might increase the holding costs at the destination. Hence, there is a need to coordinate transportation, ordering and inventory decisions in order to maximize a firm’s savings. For this study, we consider a problem of ordering multiple items from a single source by a downstream SC member, such as a retailer in a distribution environment or a manufacturer in production environment subject to uncertainty in demand. The downstream SC member makes replenishment decisions for items with stochastic demand and transportation cost over a finite planning horizon. The costs considered include holding and ordering costs for each item, along with the transportation cost for shipping the items from the vendor. In a risk-neutral approach, the objective of the problem is to minimize `here-and-now' first-stage replenishment costs and expected inventory and lost sales costs while serving customer demand from shipments and inventory. This problem is a generalization of many dynamic demand lot sizing problems in the literature, and is often encountered in retail and production environments.

In the last few decades this coordinated replenishment problem has received considerable attention. However, typically, the transportation cost is considered to be a constant with no quantity discounts. On the other hand, in SCs lower freight rates may be obtained by consolidating shipments to take advantage of economies of scale and quantity discounts. The freight rate per pound decreases as the total weight of the cargo increases, clearly providing an incentive for the shipper to consider consolidating order shipments. We assume a generic transportation cost structure that fairly represents both less-than-truckload (LTL) and full-truckload (FTL) environments. The FTL cost structure represents shipping in a containerized shipment of various sizes, while LTL cost structure represents shipping in partial loads. These freight rates could vary from week to week depending on the capacity available in the marketplace. We propose a new two-stage mathematical formulation and extend it to consider risk-measures.

The contributions of this work to the literature on optimization with risk-measures on SP2 include the following: a) solution methodology for SP2 with ASD risk measure; b) formulation and study for MSSOP in supply chain with ASD risk measure; c) demonstration of efficacy of the proposed methodology using knapsack instances from the literature. The rest of this paper is organized as follows: in the next section, relevant research from the literature is reviewed. In Section 3, problem formulations and required background for the methodology are presented. In Section 4, the foundation and algorithm are presented for the SP2 with ASD risk-measure. Section 5 introduces formulation for MSSOP in supply chain, and computational results using ASD and expected cost solutions are compared. Furthermore, the computational results based on knapsack instances are reported. Finally, Section 6 draws conclusions.

\section{Related Work}\label{sec-intro}
Research in SP2 with mean-risk measures in the literature include its characterization and algorithms for each specification. However, computational study of the empirical behavior of the algorithms is fairly limited. For a recent survey on mean-risk stochastic programs, we refer the interested reader to work of Krokhmal et al. \cite{krokhmal2011modeling}. We should also point out that algorithms for mean-risk SP2 with integer restrictions on decision variables is an ongoing active area of research [M{\"a}rkert, Andreas and Schultz \cite{15markert2005deviation}, Schultz and Tiedemann \cite{24schultz2003risk}]. Risk-averse SP2 with mean-risk measures has a significant potential to be applied in different fields. Since it is a fairly recent development, it has only been applied to just a few problems. For instance, Schultz and Tiedemann \cite{5schultz2006conditional}, Schultz and Neise \cite{12schultz2007algorithms}, Liu et al. \cite{13liu2009two}, and Noyan \cite{14noyan2012risk}, study problems in chemical engineering, energy optimization, transportation network protection, and disaster management, respectively. Similarly, Beier et al. \cite{beier2015nodal}, and Corolli et al. \cite{corolli2015two} study problems in risk-neutral setup in supply chain, and air-traffic flow management, respectively.

MSSOP belongs to the class of dynamic demand joint/coordinated capacitated lot-sizing problems (CCLSP). Robinson, Narayanan, and Sahin \cite{robinson2009coordinated} provide a taxonomy of replenishment problems in a deterministic dynamic demand environment. In CCLSP, a joint/family setup cost is incurred every time one or more items of a product family is replenished, subject to a capacity constraint and a unit cost for each item being ordered. Removing the joint setups yield multi-item capacitated lot sizing problems, and relaxing the capacity constraint yields the coordinated uncapacitated lot-sizing problem (CULSP). Restricting the number of items yields the single item capacitated lot sizing problem, and relaxing the capacity constraint yields the most elemental single item uncapacitated replenishment problem solved using Wagner and Whithin \cite{wagner1958dynamic} algorithms. The CULSP is shown to be NP-hard in Arkin, Joneja, and Roundy \cite{arkin1989computational}; hence the more general case of CCLSP is also a NP-hard problem. The literature on solving these deterministic lot-sizing problems is extensive, and the solution approaches include dynamic programming algorithms (Silver \cite{silver1979coordinated}), B\&B approaches (Robinson and Gao \cite{robinson1996dual}), Lagrangian decomposition methods (Rizk, Martela, and Ramudhinb \cite{martel2002lagrangean}; Ertogral \cite{ertogral2008multi}) and heuristic procedures (Narayanan and Robinson \cite{narayanan2010efficient}, Venkatachalam and Narayanan \cite{venkatachalam2015efficient}). A more detailed survey of these problems is provided in Zoller and Robrade \cite{zoller1988efficient}), and Jans and Degraeve \cite{jans2008modeling}.

Traditionally, the joint setup cost is considered to be a binary (0 or 1) decision variable, i.e. the entire transportation or production capacity is available for replenishment once the setup cost is incurred. In this research, we have a more general piecewise joint setup cost structure. Specifically, we assume the FTL carrier charges a fixed amount per truck whether it is partially or fully loaded, while the LTL carrier charges an incrementally reduced price for each unit shipped. The problem addressed here, not only has all the complexities of the capacitated joint replenishment problem, but also a generic piecewise family setup cost function tied to the capacity available for replenishment, making the problem even more difficult to solve. This is the most general class of dynamic lot sizing problem, involving multiple items, capacity restrictions and a common replenishment cost structure with stochastic demand. Therefore, an efficient solution approach to this problem would be of significant benefit to the SC industry.

\section{Problem Formulation}\label{sec-intro}
A SP2 problem involves optimizing the here-and-now (first-stage) costs plus expected future (second-stage) costs under risk neutral condition. A two-stage program under risk neutral (minimizing expected costs) measure is represented as follows:

\begin{equation}
\begin{alignedat}{2}\label{eq1}
\mathcal{Q}_{E} = \Min      \   & c^\top x + \mathbb{E}_\omega \phi(x,\om)          && \\
\text{s.t. } & Ax \geq b  && \\
& x \in X.  
\end{alignedat}
\end{equation}

\begin{equation}
\begin{alignedat}{2}\label{eq-2}
\phi(x,\om) = \Min \ & q(\omega)^{\top} y(\omega)            && \\
\text{s.t. } & W(\omega) y({\omega}) \geq h({\omega}) - T({\omega}) x && \\
& y(\omega) \in Y.
\end{alignedat}
\end{equation}

In the second-stage (scenario) problem, $y(\om)$ denotes the recourse decision vector and $W(\om) \in \mathbb{R}_+^{m_2 \times n_2}$ is the random recourse matrix. It is assumed that $T(\om):\Omega \mapsto \mathbb{R}^{m_2 \times n_1} $, $h(\om):\Omega \mapsto \mathbb{R}^{m_2}$ and $q(\om):\Omega \mapsto \mathbb{R}^{n_2}$ are measurable mappings defined on a probability space $(\Omega,\mathcal{F},\mathbb{P})$. The set $Y$ imposes integer restrictions on all or some components of $y(\om)$. $\omega$ is a realization of a multivariate random variable $\tilde{\omega}$, and $\mathbb{E}_\omega$ denotes the mathematical expectation operator. In the above formulation, $p_\omega$ denotes the probability of occurrence for the scenario $\omega$, and $\sum{p_{\omega \in \Omega}} = 1$, where $\Omega$ is the set of scenarios.

We consider problem \eqref{eq1} under the following assumptions:
\begin{itemize}
	\item[\textbf{A1.}]
	The random variable $\tilde{\om}$ is discrete with finitely many scenarios $\om \in \O$, each with probability of occurrence $p(\om)$ such that $\sum_{\om \in \O}p_{\omega} = 1$.
	\item[\textbf{A2.}] The first-stage feasible set $\{A x \geq b, x \in X\}$ is nonempty.
	\item[\textbf{A3.}] The second-stage feasible set $\{W y(\om) \geq h(\omega) - T(\omega) x, y(\om) \in Y\}$ and is bounded and nonempty for all $x \in \{A x \geq b, x \in X\}$.
\end{itemize}

\noindent Assumption (A1) is needed for tractability while assumptions (A2) and (A3) are needed to guarantee that the problem has an optimal solution. Assumption (A3) implies the relatively recourse assumption, i.e., $\sum_{\omega \in \Omega} p_\omega q(\omega)^\top y(\omega) < \infty$ for all $x \in \{A x \geq b | x \in X\}$. 

\begin{defn}
Expected excess risk measure for the random variable $\tilde{\omega}$ is defined as

\begin{equation}
	\begin{alignedat}{2}\label{eq-ee1}
		EE(\tilde{\omega}) = \mathbb{E} ([\tilde{\omega}- \eta]_+), && \\
	\end{alignedat}
\end{equation}
where $[a]_+ = \text{max}(a,0), a\in \mathbb{R}, \text{ and } \eta \in \mathbb{R} \text{ is a given parameter}.$ 
\end{defn}

\begin{prop}
	Based on \text{\textbf{A1}} and $\rho \in [0,1]$, a mean-risk problem with expected excess risk measure can be represented $ $  as follows: 
	
\begin{equation}
\begin{alignedat}{2}\label{prop32eq1}
\Min \  & f_{\tilde{\omega}} + \rho(EE(f_{\tilde{\omega}})), && \\
\end{alignedat}
\end{equation}
where $ f_{\tilde{\omega}} = c^\top x + \mathbb{E}_\omega \phi(x,\om)$, and this can be equivalently stated as a following problem:
	
	\begin{equation}
	\begin{alignedat}{2}\label{eq-1b}
	\mathcal{Q}_{EE} = \Min \,\,  & (1+\rho) c^\top x + \sum_{\omega \in \Omega} p_\omega q(\omega)^\top y(\omega) +  \rho	\sum_{\omega \in \Omega} p_\omega v(\omega)	   && \\
	\text{s.t. }        & A x \geq b                             && \\
	&          T(\omega) x + W(\omega) y(\omega) \geq h(\omega),   \quad \quad \quad \forall \, \omega \, \in \Omega,      && \\
	&          q(\omega)^\top y(\omega) - \eta \leq v(\omega),  \quad \forall \, \omega \, \in \Omega,  && \\	
	& x \in X, y(\omega) \in Y, v(\omega) \in \mathbb{R}^+.               &&
	\end{alignedat}
	\end{equation}
\end{prop}

\begin{proof}
EE($\tilde{\omega}$) is a coherent risk measure (\cite{schultz2010risk}). Hence, using the transnational invariance property of coherent risk measure,
	\begin{equation}
	\begin{alignedat}{2}	
		EE(c^\top x + \mathbb{E}_\omega \phi(x,\om)) = c^\top x + EE(c^\top x + \mathbb{E}_\omega \phi(x,\om)). && \nonumber
	\end{alignedat}		
	\end{equation}
Hence, $f_{\tilde{\omega}} + \rho(EE(f_{\tilde{\omega}})) = f_{\tilde{\omega}} + \rho(EE(c^\top x + \mathbb{E}_\omega \phi(x,\om)))$ \newline		
\phantom{Hence, $f_{\tilde{\omega}} + \rho(EE(f_{\tilde{\omega}}))$}	$ = (1+\rho) c^\top x + \mathbb{E}_\omega \phi(x,\om) + \rho(EE(\mathbb{E}_\omega \phi(x,\om))). \\$	

Since $\tilde{\omega}$ is a discrete random variable with finite realizations then the expectation operator for second-stage can be represented as $\mathbb{E}_\omega \phi(x,\om) = \sum_{\omega \in \Omega} p_\omega q(\omega)^\top y(\omega)$. Hence, the assertion based on \eqref{eq-2} and \eqref{eq-ee1}.
\end{proof}

\begin{defn}
	Absolute semi-deviation risk measure for the random variable $\tilde{\omega}$ is defined as
	
	\begin{equation}
	\begin{alignedat}{2}\label{eq-ee}
	ASD(\tilde{\omega}) = \mathbb{E} ([\tilde{\omega}- \mathbb{E}[\tilde{\omega}]]_+), && \\
	\end{alignedat}
	\end{equation}
	where $[a]_+ = \text{max}(a,0), a\in \mathbb{R}.$ 
\end{defn}

\begin{prop}
	Based on \text{\textbf{A1}} and $\rho \in [0,1]$, a mean-risk problem with absolute semi-deviation risk measure can be represented $ $  as follows:
	
	\begin{equation}
	\begin{alignedat}{2}\label{propeq1}
	\Min \  & f_{\tilde{\omega}} + \rho(ASD(f_{\tilde{\omega}})), && \\
	\end{alignedat}
	\end{equation}
	where $ f_{\tilde{\omega}} = c^\top x + \mathbb{E}_\omega \phi(x,\om)$, and this can be equivalently stated as a following problem:
	
	\begin{equation}
	\begin{alignedat}{2}\label{eq-1c}
	\mathcal{Q}_{ASD}= \Min \,\,  & (1-\rho)c^\top x + (1-\rho)\sum_{\omega \in \Omega} p_\omega q(\omega)^\top y(\omega) +  \rho	\sum_{\omega \in \Omega} p_\omega v(\omega)	   && \\
	\text{s.t. }        & A x \geq b                             && \\
	&          T(\omega) x + W(\omega) y(\omega) \geq h(\omega), \quad \quad \quad \forall \, \omega \, \in \Omega,         && \\
	&          c^\top x + q(\omega)^\top y(\omega) \leq v(\omega), \quad \quad \quad \forall \, \omega \, \in \Omega,    && \\	
	&          c^\top x + \sum_{\omega \in \Omega} p_\omega q(\omega)^\top y(\omega) \leq v(\omega), \quad \forall \, \omega \, \in \Omega,    && \\		
	& x \in X, y(\omega) \in Y, v(\omega) \in \mathbb{R}.               &&
	\end{alignedat}
	\end{equation}	
\end{prop}

\begin{proof}
	Similar to previous preposition.
\end{proof}

A modified formulation of expected excess as a two-stage stochastic program is represented as follows:

\begin{subequations}\label{eq-master-1}
	\begin{align}
	\mathcal{Q}'_{EE} = \Min      \   & (1-\rho) c^\top x + \theta           \\
	\text{s.t. } & Ax \geq b  \\
	& \alpha^{t \top} x + \theta \geq \beta^{t} , \ t \in 1,...,k \label{eq-master-1a} \\
	& x \in X.  \notag
	\end{align}
\end{subequations}
At a given iteration $k$ of the L-shaped algorithm the master problem denoted $MP_{ee}$, takes the form shown in \eqref{eq-master-1}. Constraints \eqref{eq-master-1a} are the \textit{optimality} cuts, where the cut co-efficients `$\alpha$' and right-hand side `$\beta$' are computed based on the optimal dual solution of the subproblems, and `$\theta$' is an unrestricted Benders' variable. Optimality cuts approximate the value function of the second-stage subproblems. For a first-stage solution $x^k$ from the master problem \eqref{eq-master-1}, the subproblem for each scenario $\om \in \Omega$, denoted as $SP_{ee}(x,\om)$, is given as follows:
\begin{subequations}\label{eq-suSIP2}
	\begin{align}
	f^k_{EE}(x^k,\omega) = \Min \ & (1-\rho) q(\omega)^\top y(\omega) +  \rho v(\omega)  \\
	\text{s.t. } & T(\omega) x^k + W(\omega) y(\omega) \geq h(\omega) \\
	&  c^\top x^k + q(\omega)^\top y(\omega) - \eta \leq v(\omega)  && \\		
	& y(\omega) \in Y.               \notag
	\end{align}
\end{subequations}
It should be noted that in order to get dual solutions to construct an optimality cut for $MP_{ee}$, the integrality restrictions for $y(\omega)$ in $SP_{ee}(x,\om)$ should be relaxed, i.e, $y(\omega) \geq 0$ instead of $y(\omega) \in Y$. We will call the relaxed version of $SP_{ee}(x,\om)$ as $SP_{ee}^{lp}(x,\om)$.
\\
\\
Schultz \cite{schultz2010risk} provided the following relationship. Let $\eta \leq \mathcal{Q}_{E}$, then the following is an established relation for a fixed $x \in X$: 

\begin{equation}
\begin{alignedat}{2}\label{eq-1e1}
\mathcal{Q}_{E}(x) \leq \mathcal{Q}'_{EE}(x) + \rho\eta \leq \mathcal{Q}_{ASD}(x).
&& \\
\end{alignedat}
\end{equation}
\\
The relationship indicates that with a good estimation for $\eta$ in $\mathcal{Q}'_{EE}(x)$ can provide a good lower bound for ASD. $\mathcal{Q}'_{EE}(x)$ has block angular structure which is amenable to traditional L-shaped method for SP2. 

\section{Decomposition Algorithms}\label{sec-intro}
In this section, we provide the details of the algorithms. The primary risk-measure algorithm for ASD is referred as `RM-ASD' algorithm, and a cut generation process is utilized within RM-ASD which we refer it as `MOD-EE-SC' procedure. We will start with the details of RM-ASD followed by which MOD-EE-SC procedure is presented. 

L-shaped algorithm framework is used to the solve modified expected excess problem with an initial estimation for target $\eta$ from expected value model. Then the target is modified based on its proximity from the expectation of the scenarios. Additionally, we add sub-gradient cuts to the master problem of modified expected excess to tighten the gap between the bounds provided in \eqref{eq-1e1}. In Step 0, we solve expected value problem and retain the solution vector as $\hat{x}$. The solution vector $\hat{x}$ is used to obtain the values for $\mathcal{Q}'_{EE}$ and $\mathcal{Q}'_{ASD}$ which are further used in \eqref{eq-1e1} to obtain lower and upper bounds for ASD risk-measure in Step 1. The current target for modified expected excess $\eta$ is evaluated and if it is an over estimate then the value of $\eta$ is reduced by a user defined parameter $\xi$ or if it is an under estimate then the value of $\eta$ is increased by $\xi$ in Step 2. In Step 3, based on the works of Ahmed \cite{1ahmed2006convexity}, and Cotton and Ntaimo \cite{cotton2015computational}, a cut is generated using the sub-gradient information of objective function for $\mathcal{Q}'_{ASD}$ using MOD-EE-SC procedure to tighten the gap between the bounds. Modified expected excess problem is solved for the adjusted $\eta$ and the first-stage solution vector is used to compute $\mathcal{Q}'_{ASD}$ in Step 4. Further, the bounds are re-evaluated and adjusted. Finally, the procedure is exited based on a user defined tolerance or looped back to Step 2. It should be noted that solving \eqref{eq-master-1} is not trivial due to the presence of binary variables in the second-stage. Typically, specialized algorithms are to be devised and one of which is utilized in our computation experiments.

\begin{figure}[!htb]
	\begin{hanglist}
		\item \textbf{RM-ASD Algorithm}
		\item \hrulefill
		
		{\it\bf \item Step 0. Initialization.}  Solve \eqref{eq1} and let the solution for $x$ be $\hat{x}$. Set $\eta \leftarrow \mathcal{Q}_{E}$, $\epsilon > 0,\xi > 0, LB \leftarrow -\infty, UB \leftarrow \infty$. Let $S^+$ and $S^-$ be empty sets.
		
		{\it\bf \item Step 1. Generate bounds using expected value problem.} Set $x = \hat{x}$ to solve \eqref{eq-master-1} and obtain $\mathcal{Q}'_{EE}$ and $\mathcal{Q}_{ASD}$. Set $LB = \mathcal{Q}'_{EE} + \rho \eta$ and $UB = \mathcal{Q}_{ASD}$.
		
		{\it\bf  \item Step 2. Evaluate and adjust the target $\eta$ for modified expected excess problem.} 
\newline \textit{Step 2a.} For $\omega \in \Omega$: If {$c^\top \hat{x} + q(\omega)^\top y(\omega) -  \eta > 0 $}, then $S^+ = S^+ \cup \omega$, else if {$c^\top \hat{x} + q(\omega)^\top y(\omega) -  \eta < 0 $}, then $S^- = S^- \cup \omega$.		
\newline \textit{Step 2b.} If {$p_\om(|S^+|) > p_\om(|S^-|)$}, then $\eta = \eta + \xi$, else if {$p_\om(|S^+|) < p_\om(|S^-|)$}, then $\eta = \eta - \xi$. 
		
		{\it\bf  \item Step 3. Generate sub-gradient based cut for modified excess problem.} Use MOD-EE-SC procedure to generate sub-gradient based cuts for modified excess problem and add them to the master problem \eqref{eq-master-1}.
		
		{\it\bf  \item Step 4. Compare and updated the bounds.} Solve \eqref{eq-master-1} to obtain $\mathcal{Q}'_{EE}$ and let the solution be $\hat{x}$. Set $x=\hat{x}$ and compute $\mathcal{Q}_{ASD}.$ If {$\mathcal{Q}'_{EE} + \rho \eta > LB$}, then set $LB = \mathcal{Q}'_{EE} + \rho \eta$. If {$\mathcal{Q}_{ASD} < UB$}, then set $UB = \mathcal{Q}_{ASD}$.

		{\it\bf  \item Step 5. Termination.} If {$UB - LB < \epsilon$}, then stop, else initialize $S^+$ and $S^-$ and go to Step 2.						

		\item \hrulefill
	\end{hanglist}
	\caption{Risk-Measure - Absolute semi-deviation (RM-ASD) Algorithm}
	\label{figRMASD}
\end{figure}

For MOD-EE-SC procedure, we initialize the parameters in Step 0. The relaxed sub-problems $SP_{ee}^{lp}(x,\om)$ for modified expected excess problem are solved at a given $k$ iteration, and a mean value `$\bar{Q}$' is calculated based on the objective values of second-stage problems in Step 1. The deviation, cut co-efficients and right-hand side value are calculated based on the gradient information from the objective functions in Step 2. Finally, a cut is generated and added to $MP_{ee}$ where `$\gamma$' is an unrestricted variable.  

\begin{algorithm}
	\caption*{MOD-EE-SC Procedure}
	\label{alg:FCG1}
	\baselineskip 0.5 cm
	\small
	\begin{algorithmic}
		\vspace{.1cm}
		\STATE \textbf{Step 0. Initialization:} Set $\sigma \leftarrow 0, \sigma^0 \leftarrow 0, \bar{Q} \leftarrow 0, $ $x \in X$.
		\STATE \textbf{Step 1. Solve Subproblem LPs:}
		\FOR{$\omega \in \Omega$ solve subproblem $SP_{ee}^{lp}(x,\om)$} 
		\STATE get dual solution $\bar{\pi}(\om)$
		\STATE compute $\bar{Q} \leftarrow \bar{Q} + p(\om) f^k_{EE}(x^k,\omega)$
		\ENDFOR					
		\STATE \textbf{Step 2. Generate cut co-efficients and RHS value:}		
		\FOR{$\omega \in \Omega$} 
		\IF { $f^k_{EE}(x,\omega) \geq \bar{Q} $}
		\STATE compute $\sigma \leftarrow \sigma + p(\om)[T(\om)^\top\bar{\pi}(\om)-c]$
		\STATE compute $\sigma^0 \leftarrow \sigma^0 + p(\om)[\bar{\pi}(\om)^\top r(\om)]$		
		\ELSE 		
		\STATE compute $\sigma \leftarrow \sigma + p(\om)[\sum_{\omega' \in \Omega} p(\om')T(\om')^\top\bar{\pi}(\om')-c]$
		\STATE compute $\sigma^0 \leftarrow \sigma^0 + p(\om)[\sum_{\omega' \in \Omega}\bar{\pi}(\om')^\top r(\om')]$		
		\ENDIF					
		\ENDFOR
		\STATE \textbf{Step 3. Update and Solve the Master Problem:}
		\STATE add optimality cut $\sigma^\top x + \gamma \geq \sigma^0$ to master problem \eqref{eq-master-1}.		
	\end{algorithmic}
\end{algorithm}

\section{Computational Study}\label{sec-intro}
In this section, we provide results from two types of computational experiments. Our objective is to provide insights with regard to decision making process when ASD is considered instead of using risk-neutral  expected costs, and efficacy of the algorithm presented in the previous sections for ASD models. For MSSOP, replenishment policies are generated using ASD risk-measure and risk-neutral models, and their performance and robustness of the solutions are evaluated by simulating a set of demand scenarios representing the reality. Then the efficiency of proposed methodology is tested using stochastic multi-dimensional knapsack instances from the literature.

\subsection{Multi-Item Single Source Ordering Problem formulation}\label{S1.1}

This is a traditional product unit formulation, where the production variables are represented as units of demand. Consider a $t \in T$ period planning horizon problem with $i \in I$ items and $j \in J$ discount break points. In the two-stage model, the replenishment costs comprising of minor setup and transportations costs for the entire planning horizon represents here-and-now first-stage decision. Based on the first-stage replenishment plan, the recourse decisions to minimize the expected inventory and unsatisfied costs are minimized due to randomness in demand. The first-stage model uses the following deterministic parameters, $k_i$, setup cost for item $i$; $f_{jt}$, transportation cost at freight discount break point $j$ in period $t$; $w_i$, per unit weight of item $i$; $m_j$, freight discount break point $j$; $h_i$, inventory holding cost per unit of item $i$ per period; $d_{it}(\om)$, demand for item $i$ for period $t$ for scenario $\om$. The first-stage decision variables include: $x_{it}$,  order size of item $i$ replenished in period $t$; $y_{it} = 1$ if item $i$ is replenished in period $t$; $z_t$, total weight of items ordered/transported in period $t$; $q_{jt}$, binary variable that indicates where the freight discount range $z_t$ falls, if $m_j \leq \sum_{i \in I} w_i \sum_{i \in I} x_{it} \leq m_{j+1}$ then $q_{jt}= 1$, otherwise $0$, and $z_{jt}$, auxiliary variable linking transportation quantity to the piecewise linear transportation cost. The second-stage variables include, $v_{it}(\om)$, ending inventory of item $i$ in period $t$ for the scenario $\om$; $u_{it}(\om)$, lost sales of item $i$ in period $t$ for the scenario $\om$;

The first-stage model for MSSOP to minimize the replenishment costs is represented as follows:

\begin{equation}
	\setlength{\belowdisplayskip}{0pt} \setlength{\belowdisplayshortskip}{0pt}
	\setlength{\abovedisplayskip}{0pt} \setlength{\abovedisplayshortskip}{0pt}
	\text{SCM=} \Min \sum_{i \in I} \sum_{t \in T} k_i y_{it} + \sum_{j \in F} \sum_{t \in T} f_{j} z_{jt} + \mathbb{E}_\omega \phi(x,\omega) \qquad \nonumber
\end{equation}

\begin{xalignat}{3}
	\setlength{\belowdisplayskip}{0pt} \setlength{\belowdisplayshortskip}{0pt}
	\setlength{\abovedisplayskip}{0pt} \setlength{\abovedisplayshortskip}{0pt}
	& \hspace{-.40cm}\mbox{Subject to} && \nonumber \\[0.2cm]
	& x_{it} \leq My_{it} \, 
	&& \quad \forall  \, i\in I, t\in T, \\[0.2cm]	
	& \sum_{i \in I} w_i \sum_{i \in I} x_{it} \leq \sum_{j \in J} m_j z_{jt}  \,
	&& \quad \forall  \, t\in T, \\[0.2cm]        
	& z_{1t} \leq q_{1t} \, 
	&& \quad \forall  \, t\in T, \\[0.2cm]
	& z_{jt} \leq q_{j-1t} + q_{jt} \, 
	&& \quad \forall  \, j\in 2,...,J-1, t\in T, \\[0.2cm]
	& z_{Jt} \leq q_{J-1t} \, 
	&& \quad  \forall  \, t\in T, \\[0.2cm]
	& \sum_{j \in J} z_{jt} \leq 1  \,
	&& \quad \forall  \, t\in T, \\[0.2cm]        
	& \sum_{j \in J} q_{jt} \leq 1  \,
	&& \quad \forall  \, t\in T, \\[0.2cm]        
	& q_{jt},y_{it} \in \{0,1\}, x_{it} \geq 0, z_{jt} \geq 0 && \quad \forall  \, i \in I, \, j \in J,\, t \in T. \\[0.2cm]        	    	
\end{xalignat}

The second-stage recourse function based on the replenishment decisions taken in first-stage and for a particular scenario $\om$ in $\Omega$ is represented as follows:

\begin{equation}
\setlength{\belowdisplayskip}{0pt} \setlength{\belowdisplayshortskip}{0pt}
\setlength{\abovedisplayskip}{0pt} \setlength{\abovedisplayshortskip}{0pt}
 \phi(x,\omega) =  \Min \sum_{i \in I} \sum_{t \in T} [ (h_{i} v_{it}(\om) + p_i u_{it}(\om) )] \qquad \nonumber
\end{equation}

\begin{xalignat}{2}
	& \hspace{-.40cm}\mbox{Subject to} && \nonumber \\[0.2cm]
	& v_{i1}(\om) = o_i + \sum_{t \in T} x_{it} + u_{i1}(\om) -d_{i1}(\om) \,
	&& \quad \forall  \, i\in I, \\[0.2cm]        	
	& v_{it}(\om) = v_{it-1}(\om) + x_{it} + u_{it}(\om) -d_{it}(\om) \,
	&& \quad \forall  \, i\in I, t\in T, \\[0.2cm]        	
	& v_{it}(\om) \geq 0, 0 \leq u_{it}(\om) \leq d_{it}(\om) && \quad \forall  \, i \in I, \, t \in T.     
\end{xalignat}

Constraint set (12) prohibits replenishment unless the setup charge is incurred. Constraint sets (13)-(16) model the transportation cost as a piecewise linear function. The weight range $z_t$ that is shipped in each period is indicated by $q_{jt}$ being equal to 1 ($m_j<= \sum_{i \in I} w_i \sum_{i \in I} x_{it} < m_{j+1}$). The corresponding discounted freight cost is obtained by a convex combination of the transportation cost values at two discount break points by making sure $z_{kt} + z_{k+1t} = 1$, while remaining $z_{kt}$s are set to zero. Constraints (19) present the restrictions for decision variables. Also, to obtain a tighter bound, we replace $M$ with $\sum_{t'=t}^{T} max(d_{it'}\om)$ in constraint set (13).

In the second-stage formulation, constraint sets (21) and (22) are the typical mass-balance constraints to determine the on-hand inventory for each item for the first period and other period of the planning horizon, respectively. The objective function minimizes the penalties for lost sales and inventory for each item $i$ for each period $t$. 

In this experiment set up, we compare replenishment policy created by expected costs and ASD  formulation over a planning horizon of 10 time periods. The experiments are conducted with number of items as 5, 10, and 15 for 10 time periods with 25 scenarios in each case. The experiment set up is motivated by real-time practice where replenishment decisions are made periodically in organizations. The first-stage decisions are the replenishment decisions made for the entire planning horizon based on which purchase orders are made, and the second-stage reflect the appropriate risk-measure cost functions for lost sales and inventory due to the realization of demand and replenishment decisions made in the first-stage. This way the manager of the firm will know the replenishment plans for the upcoming weeks based on which transportation and other resources can be scheduled. The replenishment decisions from risk-neutral expected costs and ASD are evaluated by simulating demand for the items for entire planning horizon. Five replications were generated for each of the instances, and expected lost sales and inventory costs were calculated. The performance of policies using expected costs and ASD are measured by two perspectives; \# of time periods the demand was unmet or lost sales occurred and the quantity of lost sales. This is a standard way to measure the fill-rate in practice. The scenarios are sampled with demand from a normal distribution $\mathcal{N}(100,10)$. As denoted in Cotton and Ntaimo \cite{cotton2015computational}, the risk-neutral approach is still appropriate for most of the standard stochastic programming test instances due to their uniform or normal-like marginal distributions. However, when the distributions are modified, the risk-neutral approach may no longer be appropriate and the risk-averse approach becomes necessary. Hence, 20\% of the demand were considered lumpy, i.e, with normal distribution $\mathcal{N}(150,20)$. The transportation cost parameters used for all the instances are as follows: $f_1=0, f_2 = f_3 = 1,000, f_4 = f_5 = 1,500, f_6 = f_7 = 2,200; m_1 = 0, m_2 = 0, m_3 = 17,500, m_4 = 17,500, m_5 = 35,000, m_6 = 35,000,$ and $m_7 = 70,000$. The values for $f$ and $m$ are devised based on the demand distribution, number of items and time periods. The holding cost, fixed item setup cost, and unit weights are generated based on Venkatachalam and Narayanan \cite{venkatachalam2015efficient} from uniform distributions, $\mathcal{U}(50,100), \mathcal{U}(500,1000)$ and $\mathcal{U}(1,5)$, respectively. The simulation runs use similar distributions for demand.

\begin{figure}[H]
	\centering
	\includegraphics[scale=1.0]{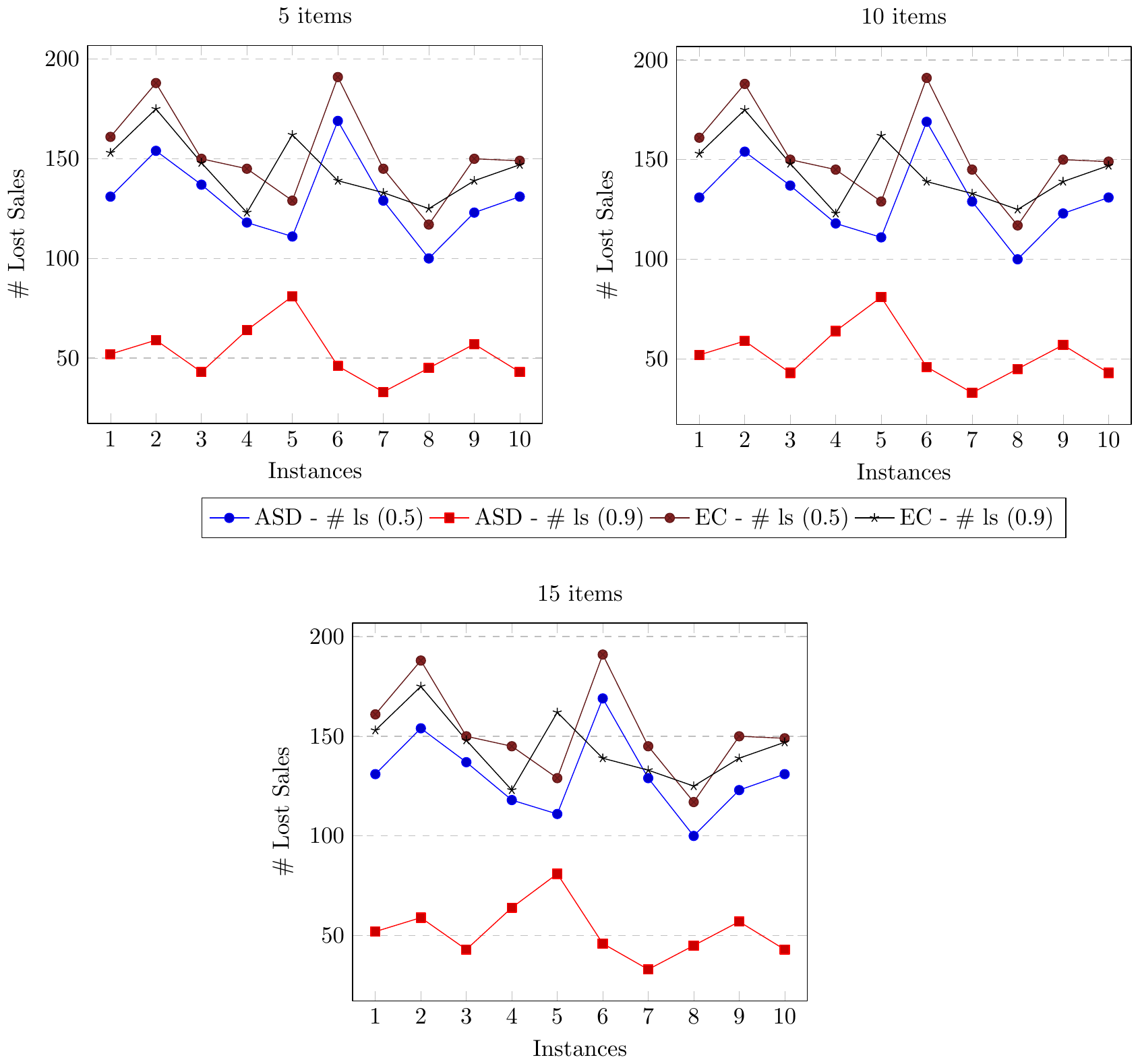}\\
	\caption{Number of Lost Sales from simulation study}
	\label{sc-fig1}
\end{figure}

Figures \ref{sc-fig1}, \ref{sc-fig2}, and \ref{sc-fig3} represent the number of occurrences and quantity of lost sales, and total replenishment costs from the simulation runs, respectively. In the figures, `ASD-\# ls(0.5)', `ASD-\# ls(0.9)' represent the lost sales count using ASD risk-measure with $\rho$=0.5 and $\rho$=0.9, respectively. Similarly, `EC-\# ls(0.5)', `EC-\# ls(0.9)' represent the count for lost sales using expected costs compared with ASD risk-measure with $\rho$=0.5 and $\rho$=0.9, respectively. The results indicate that SCM model with ASD risk measure and $\rho=0.9$ has the least number of lost sales followed by ASD risk measure and $\rho=0.5$. The runs from risk neutral expected costs have the highest lost sales. The trend is similar in the figure \ref{sc-fig2} for lost sales quantity as well.

\begin{figure}[H]
	\centering
	\includegraphics[scale=1.0]{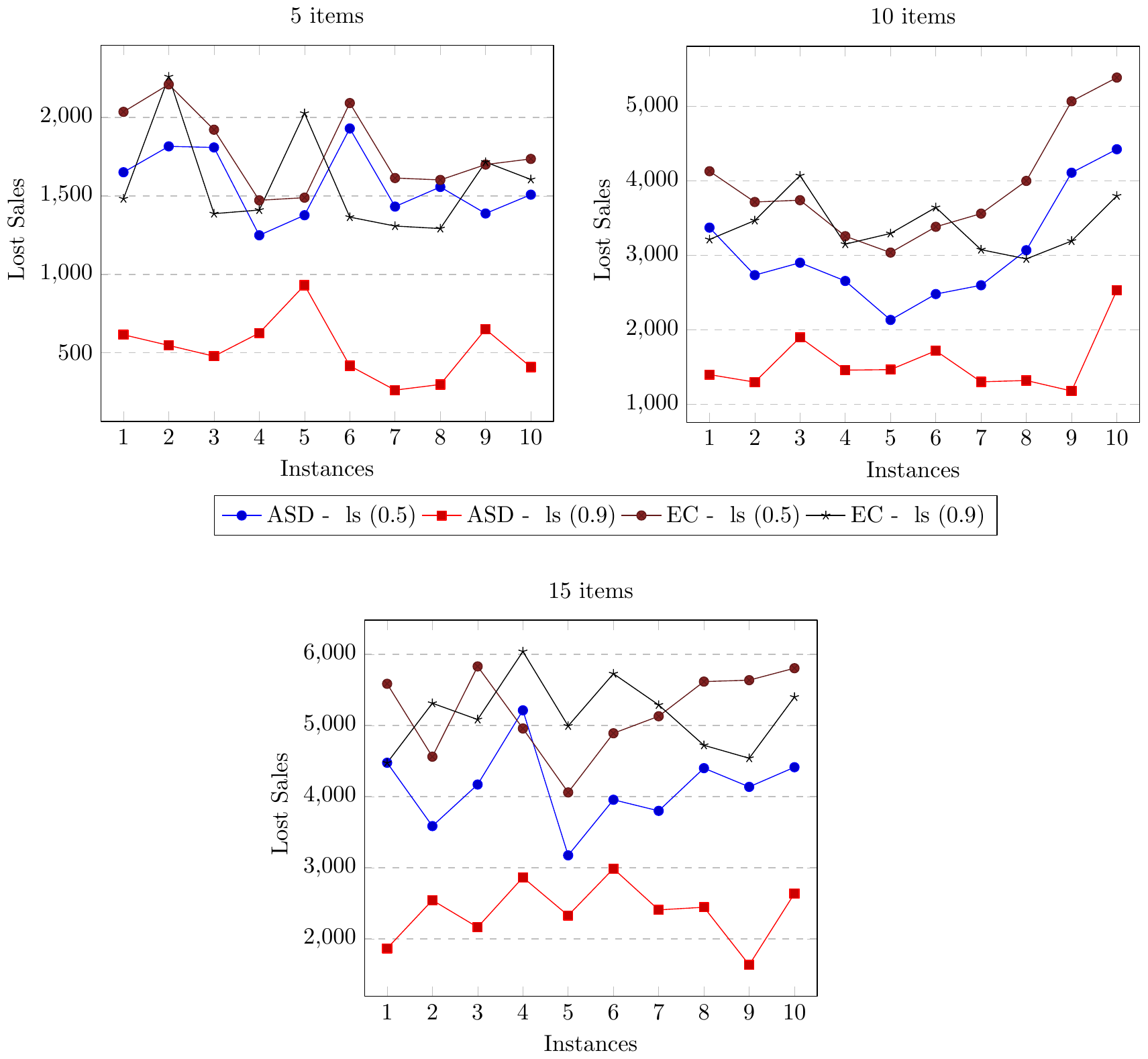}\\
	\caption{Quantity of Lost Sales from simulation study}
	\label{sc-fig2}
\end{figure}

Figure 7 represents the replenishment costs and penalty cost from excess inventory or lost sales using ASD and risk-neutral model. In the figure, `ASD-ls(0.9)', `EC-ls(0.5)', and `ASD-ls(0.5)' represent costs for ASD with $\rho=0.9$, risk-neutral expected costs, and ASD with $\rho=0.5$, respectively.  As expected costs are least expensive followed by ASD with $\rho=0.5$, and finally ASD with $\rho=0.9$. The costs for ASD with $\rho=0.9$ are the highest as these are the most conservative replenishment policies due to higher weightage for the variability in demand.

It should be noted that with an average increase of 0.3\% in costs for 5 items replenishment policy for ASD with $\rho=0.5$ compared to risk-neutral policy, it decreases lost sales occurrences by 14\% and 12\% in lost sales quantity on average. Similarly, for 10-items and 15-items data instances, replenishment policies with ASD and $\rho=0.5$ are more expensive than risk-neutral policies by 0.6\% and 0.8\% on an average, respectively. However, the occurrences of lost sales for 10-items and 15-items decreases by 9.8\% and 22.2\% on average, respectively. The lost sales quantity for 10-items and 15-items also significantly decreases by 22\% and 20\% on average, respectively.

The trend is similar for the comparison between ASD with $\rho=0.9$ and risk-neutral approaches. With an average increase of 8.7\% in costs for 5 items replenishment policy for ASD with $\rho=0.9$ compared to risk-neutral policy, there is a decrease of 65\% on average in lost sales occurrences and 70\% on average in lost sales quantity. Similarly, for 10-items and 15-items data instance, replenishment policy with ASD and $\rho=0.9$ are more expensive than risk-neutral policies by 7.7\% and 7.2\% on average, respectively. However, the lost sales occurrences for 10-items and 15-items decreases by 49\% and 50\% on average, respectively. The lost sales quantity for 10-items and 15-items also significantly decreases by 60\% and 54\% on average, respectively.  

\begin{figure}[H]
	\centering
	\includegraphics[scale=0.85]{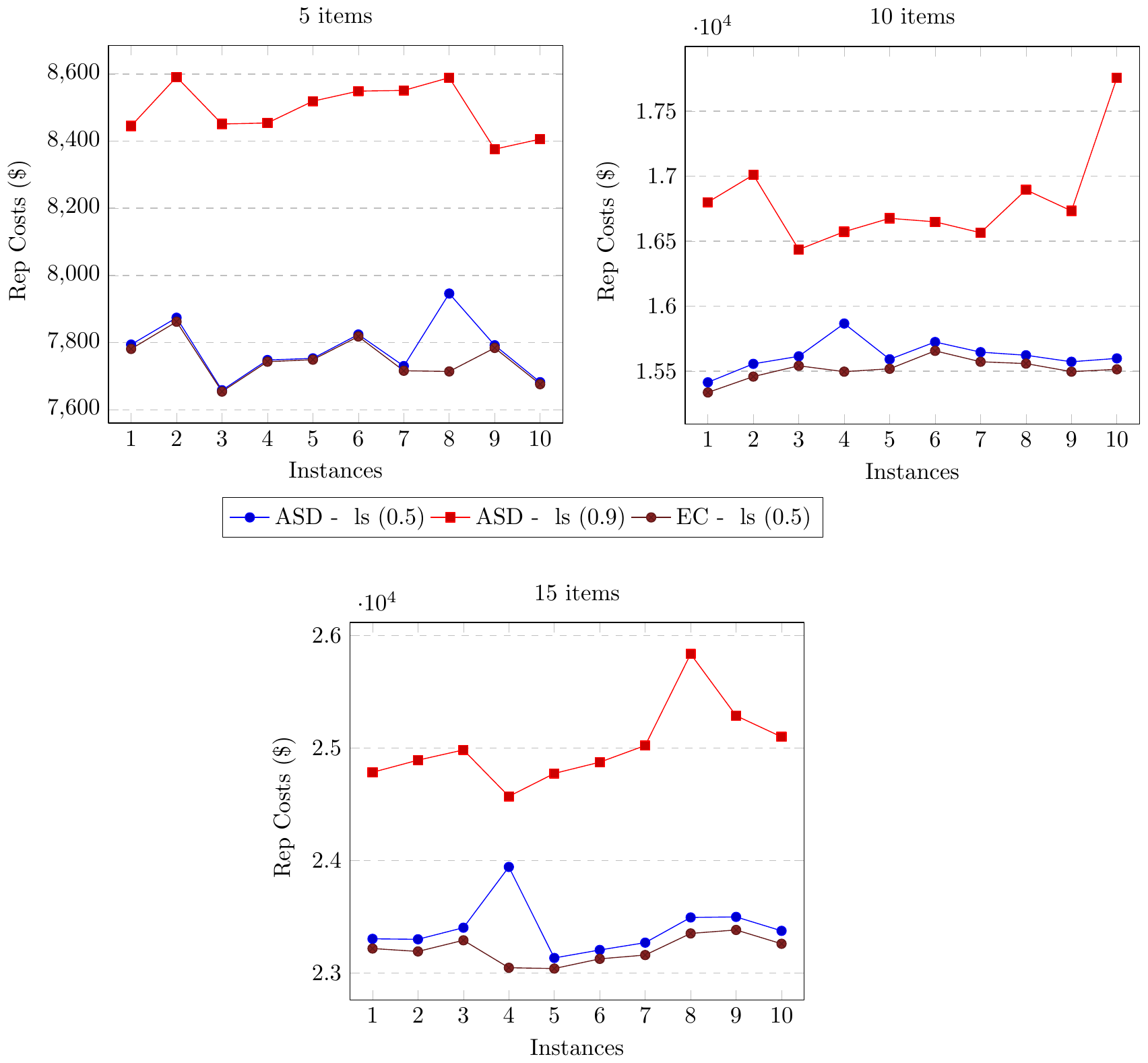}\\
	\caption{Replenishment Costs from simulation study}
	\label{sc-fig3}
\end{figure}

The results clearly provide insights of different policies from cost and performance perspectives. In our experiment setups, using the ASD model with $\rho=0.5$ is very attractive since with a very small marginal increase of replenishment and second-stage costs, it was able to provide significant improvement in lost sales count and quantity compared to risk-neutral approach.

\subsection{Stochastic Multidimensional Knapsack Problems Test Sets}\label{subsec-knapsack}
General knapsack constrained stochastic programs have received attention in the literature. Knapsack constraints appear in many applications of SP2 such as investment planning, transportation, scheduling, selling of assets and investment selection, and operations strategy. The exact stochastic multidimensional knapsack problem test instances used in this study are considered by Eric et al. \cite{beier2015stage} and Venkatachalam \cite{saran}. This class of SP2 can be formulated as follows:

\begin{equation}
\begin{alignedat}{2}\label{eq-knap1}
\mathcal{Q}_{E} = \Min   & \sum_{i=1}^{n_1} c_i^\top x_i +  \mathbb{E}_\omega \phi(x,\om) && \\
\text{s.t. }        & \sum_{i=1}^{n_1} x_i \leq b           && \\
& x_i \in \{0,1\}, \, \, \forall i=1\ldots n_1                               &&
\end{alignedat}
\end{equation}

In problem \eqref{eq-knap1}, $x$ denotes the first-stage decision vector, $c \in \Re^{n_1}$ is the first-stage cost vector, $b \in \Re$ is the first-stage righthand side, $ \phi(x,\tilde{\om})$ is the recourse function with $\tilde{\om}$ being a multivariate random variable, and $\mathbb{E}$ denotes the mathematical expectation operator satisfying $\mathbb{E}_\omega \phi(x,\om) < \infty$. The underlying probability distribution of $\tilde{\om}$ is discrete with a finite number of realizations (scenarios/subproblems) in set $\Omega$ and corresponding probabilities $p_\om, \om \in \Omega$. Thus for a given scenario $\om \in \Omega$, the recourse function $\phi(x,\om)$ is given by the following second-stage binary program:
\begin{equation}
\begin{alignedat}{2}\label{eq-knap2}
\phi(x,\om) = \Min         & \sum_{i=1}^{n_2} q(\om)^{i\top} y(\om)^{i}       && \\
\text{s.t. } & \sum_{i=1}^{n_2} w^{ij} y(\om)^{i} \leq h(\om)^j - \sum_{i=1}^{n_1} x_i, \ \forall j=1\ldots m_2 && \\
& y(\om)^i \in \{0,1\}, \ \forall i=1\ldots n_2.               &&
\end{alignedat}
\end{equation}
In formulation \eqref{eq-knap2}, $y(\om)$ is the recourse decision vector,  $q(\om) \in \Re^{n_2}$ is the recourse cost vector, $w \in \Re^{m_2 \times n_2}$ is the recourse parameter, and $h(\om) \in \Re^{m_2}$ is the righthand side.

This formulation has knapsack constraints in both the first- and second-stages, and each subproblem has equal probability of occurrence. Instance data were randomly generated using the uniform distribution ($\mathcal{U}$) with different parameter values. The knapsack weights were generated by sampling from $\mathcal{U}(2, 8)$. Objective function coefficients were generated with the first-stage costs being chosen to be much higher than second-stage costs. Objective function coefficients for first-stage variables were sampled from $\mathcal{U}(400, 650)$ while those for the second-stage were sampled from $\mathcal{U}(6, 16)$. To generate tight knapsack constraints, the righthand side value for each of the constraints was generated by finding the maximum knapsack weight ($W_{max}$) for the constraint and sampling from $\mathcal{U}(2+2W_{max}, 4W_{max})$.

In the algorithm \ref{figRMASD}, we need to determine $\mathcal{Q}_{E}$ and $\mathcal{Q}'_{EE}$, since the second-stage problem has integrality restrictions, we used Fenchel decomposition procedure. The details of Fenchel decomposition can be found at Beier et al. \cite{beier2015stage}, Venkatachalam \cite{saran}, and Venkatachalam and Ntaimo \cite{venkatachalam2016integer}. We considered four test sets, each with five randomly generated instances of same size. The problem characteristics are given in Table \ref{tab:kpdimension}. The columns of the table are problem name, `Scens' is the number of scenarios, `Bvars' is the number of binary variables, `Constr' is the number of constraints, and `Nzeros' is the number of non-zero elements for each of the problem instances. The first numeral in the problem name describes the number of first-stage variables, the second describes the number of second-stage variables, and the third describes the number of scenarios. Computational results are reported in Tables \ref{tab:comp1} and \ref{tab:comp2} reporting results for Sets 1 and 2, and Sets 3 and 4, respectively. For each instance, five different replications were executed, with one hour time limit. The columns of the tables are organized as follows:`Instance' is the instance name and the following three columns are based on the runs from RM-ASD algorithm. `LB' is the lower bound of the algorithm, `UB' is the upper bound of the algorithm, and `Gap(\%)' is the gap between the LB and UB value after the stipulated runtime of one hour. Finally, `Obj-C' is the objective and `C Gap(\%)' is the CPLEX MIP gap after solving deterministic equivalent problem (DEP) for one hour.

\begin{table}[!ht]
	\begin{center}
		\small
		\begin{tabular}{|l|c|c|c|c|}
			\hline
			Problem    & Scens & Bvars & Constr & Nzeros \\
			\hline
			\hline
			K.10.20.50      & 50 & 1,010 & 1,010 & 30,100\\
			K.10.20.100     & 100 & 2,010 & 2,010 & 60,100\\
			K.20.30.50      & 50 & 1,520 & 1,010 & 50,200\\
			K.20.30.100     & 100 & 3,020 & 2,010 & 100,200\\
			K.30.40.50      & 50 & 2,030 & 1,010 & 70,300\\
			K.30.40.100     & 100 & 4,030 & 2,010 & 140,300\\
			K.40.50.50      & 50 & 2,540 & 1,010 & 90,400\\
			K.40.50.100     & 100 & 5,040 & 2,010 & 180,400\\
			\hline
		\end{tabular}
	\end{center}
	\caption{DEP Instance Characteristics}
	\label{tab:kpdimension}
\end{table}

\begin{figure}[H]
	\centering
	\includegraphics[scale=0.50]{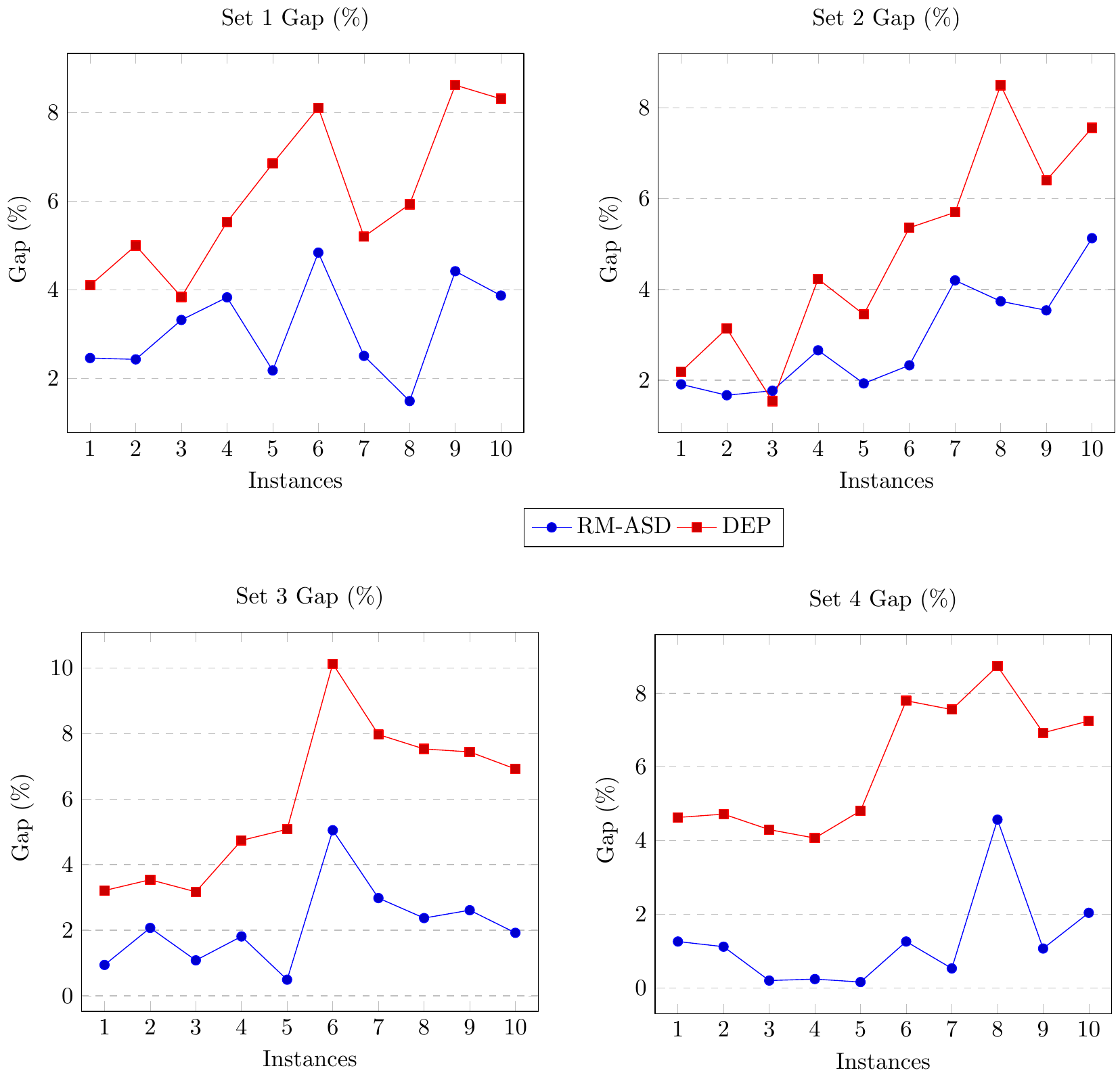}\\
	\caption{Performance of RM-ASD Algorithm}
	\label{brp-fig3}
\end{figure}

The results indicate that RM-ASD algorithm performs consistently better than direct solver. As the instance size increases, the gap between direct solver and RM-ASD algorithm widens further.
\section{Conclusions}\label{sec-intro}
Risk measures adhering to statistical functionals such as \textit{coherent risk of measures} has become a widely accepted practice in stochastic programming and optimization under uncertaintyy. The coherent risk measures are broadly classified into two categories - quantile and deviation based risk measures. This paper presents a methodology for one such `deviation' based risk measure \textit{`absolute semi-deviation' (ASD).} The methodology is evaluated from two perspectives - usefulness of ASD over risk-neutral expected costs by considering a replenishment problem in supply chain, and  efficacy of the algorithm is evaluated using standard multi-dimensional knapsack instances from the literature. The replenishment policies from ASD and expected costs for supply chain model are evaluated using a simulated replications of demand over the planning horizon. The number and quantity of lost sales are used to evaluate the policies and the results indicate for the given experiments provide insights that with ASD risk measure though the replenishment costs increases marginally overall key performance metrics for supply chain are well protected. Due to lack of consideration of variation in data by traditional risk-neutral expected costs, the computational results indicate that considering risk-measures for volatile demand will improve the performance of the system may be with a marginal increase in costs. 

Solving stochastic integer programs is a very hard prospect, and most algorithms are very dependent on the structure of the problem to guarantee tractability and convergence. For this work, the stochastic data was used to formulate the deterministic equivalent problem and then solved using an integer programming solver directly. The number of random variable realizations for each time period was limited in order to make the subproblems solvable within a reasonable amount of time. Thus, extensions to this work include the development and implementation of decomposition algorithms for solving the ASD risk measure with 0-1 integer variables in the second-stage. The lack of block angular structure for ASD makes its harder to use traditional Benders' decomposition algorithm, hence we proposed a methodology to use a modified expected excess model to get better bounds for ASD. The computational results using stochastic multi-dimensional knapsack instances indicate consistent better performance from our algorithm compared to a direct solver.    

Another extension to methodology will to consider integer variables in the second-stage where quite a lot of application like lot-sizing in supply chain will benefit. There has been recent research focus on risk measures while considering stochastic dominance, so considering risk-measure with stochastic dominance and a quantile measure like CVar and compare its performance against ASD both qualitatively and quantitatively will also provide better insights for the application.

\pagebreak
\section{Appendix}

	\begin{table}[!ht]
		\begin{center}
			\small
			\begin{tabular}{|c|c|l|c|c|c|c|c|}
				\hline
				&\textbf{Problem}    & \textbf{Instance} & \textbf{LB} & \textbf{UB} & \textbf{Gap(\%)} &\textbf{Obj-C}&\textbf{C-Gap(\%)} \\
				\hline
				&1&knaps.10.20.50.a&-89.82&-87.61&2.46&-87.61&4.10\\
				&2&knaps.10.20.50.b&-91.30&-89.08&2.43&-89.11&5.00\\
				&3&knaps.10.20.50.c&-82.99&-80.23&3.32&-80.43&3.84\\
				&4&knaps.10.20.50.d&-83.44&-80.24&3.83&-80.21&5.53\\
				Set 1&5&knaps.10.20.50.e&-88.49&-86.56&2.18&-86.46&6.85\\
				&6&knaps.10.20.100.a&-62.64&-59.60&4.84&-58.95&8.11\\
				&7&knaps.10.20.100.b&-65.15&-63.51&2.51&-63.29&5.20\\
				&8&knaps.10.20.100.c&-59.22&-58.33&1.49&-57.83&5.93\\
				&9&knaps.10.20.100.d&-59.21&-56.59&4.42&-56.18&8.62\\
				&10&knaps.10.20.100.e&-59.42&-57.12&3.87&-56.64&8.31\\
				\hline
				&Average&&&3.14&&&6.15\\
				\hline
				&1&knaps.20.30.50.a&-93.62&-91.83&1.91&-92.02&2.19\\
				&2&knaps.20.30.50.b&-91.70&-90.17&1.67&-89.99&3.14\\
				&3&knaps.20.30.50.c&-90.14&-88.54&1.77&-88.76&1.54\\
				&4&knaps.20.30.50.d&-90.44&-88.04&2.66&-88.23&4.23\\
				Set 2&5&knaps.20.30.50.e&-91.88&-90.10&1.93&-90.22&3.45\\
				&6&knaps.20.30.100.a&-68.76&-67.15&2.33&-66.79&5.36\\
				&7&knaps.20.30.100.b&-61.03&-58.47&4.20&-58.42&5.70\\
				&8&knaps.20.30.100.c&-64.77&-62.35&3.74&-61.34&8.50\\
				&9&knaps.20.30.100.d&-64.16&-61.89&3.54&-61.08&6.40\\
				&10&knaps.20.30.100.e&-60.92&-57.80&5.13&-58.18&7.56\\
				\hline
				&Average&&&2.89&&&4.81\\				
				\hline
			\end{tabular}
		\end{center}
		\caption{ Computational Results (Set 1 and Set 2) - Runtime (3600s)}
		\label{tab:comp1}
	\end{table}

	\begin{table}[!ht]
		\begin{center}
			\small
			\begin{tabular}{|c|c|l|c|c|c|c|c|}
				\hline
				&\textbf{Problem}    & \textbf{Instance} & \textbf{LB} & \textbf{UB} & \textbf{Gap(\%)} &\textbf{Obj-C}&\textbf{C-Gap(\%)} \\
				\hline
				&1&knaps.30.40.50.a&-95.22&-94.33&0.94&-94.62&3.21\\
				&2&knaps.30.40.50.b&-94.02&-92.07&2.07&-92.31&3.54\\
				&3&knaps.30.40.50.c&-93.86&-92.84&1.08&-92.94&3.17\\
				&4&knaps.30.40.50.d&-94.90&-93.18&1.81&-93.13&4.74\\
				Set 3&5&knaps.30.40.50.e&-96.09&-95.61&0.49&-95.69&5.08\\
				&6&knaps.30.40.100.a&-67.11&-63.72&5.05&-63.62&10.13\\
				&7&knaps.30.40.100.b&-68.66&-66.61&2.98&-66.17&7.97\\
				&8&knaps.30.40.100.c&-62.96&-61.47&2.37&-61.88&7.53\\
				&9&knaps.30.40.100.d&-65.73&-64.02&2.61&-64.13&7.44\\
				&10&knaps.30.40.100.e&-64.12&-62.89&1.92&-63.47&6.92\\
				\hline
				&Average&&&2.13&&&5.97\\
				\hline
				&1&knaps.40.50.50.a&-95.13&-93.94&1.26&-93.9&4.63\\
				&2&knaps.40.50.50.b&-96.04&-94.97&1.12&-95.01&4.72\\
				&3&knaps.40.50.50.c&-93.89&-93.70&0.20&-93.71&4.30\\
				&4&knaps.40.50.50.c&-96.84&-96.60&0.24&-96.61&4.07\\
				Set 4&5&knaps.40.50.50.d&-95.96&-95.81&0.16&-95.75&4.81\\
				&6&knaps.40.50.100.a&-66.61&-65.77&1.26&-65.84&7.80\\
				&7&knaps.40.50.100.b&-67.55&-67.19&0.53&-67.19&7.56\\
				&8&knaps.40.50.100.c&-68.67&-65.54&4.57&-65.54&8.74\\
				&9&knaps.40.50.100.d&-66.98&-66.26&1.07&-66.54&6.93\\
				&10&knaps.40.50.100.e&-66.44&-65.08&2.04&-65.42&7.25\\
				\hline
				&Average&&&1.24&&&6.08\\
				\hline
			\end{tabular}
		\end{center}
		\caption{ Computational Results (Set 3 and Set 4) - Runtime (3600s)}
		\label{tab:comp2}
	\end{table}

\newpage
\bibliography{asd}

\end{document}